\documentclass[11pt]{article}
\usepackage{amssymb}
\usepackage{amsfonts}
\usepackage{amsmath}
\usepackage{mathrsfs}
\usepackage{graphicx}
\usepackage{amsbsy}
\usepackage{theorem}
\usepackage{color}
\usepackage{hyperref}
\usepackage{tikz}
\usepackage[normalem]{ulem}
\usepackage{mathtools}
\newtheorem{theorem}{Theorem}[section]
\newtheorem{lemma}[theorem]{Lemma}

\newtheorem{definition}[theorem]{Definition}

\newenvironment{proof}{\noindent{\bf Proof.~}}
{{\mbox{}\hfill {\small \fbox{}}\\}}
\def\qed{\mbox{}\hfill {\small \fbox{}}\\}
\newcommand{\RR}{\mathbb{R}}
\newcommand{\R}{\mathbb{R}}
\newcommand{\ve}{\varepsilon}

\begin{document}

\title{On the  Hilbert Method in the Kinetic Theory of Multicellular Systems: Hyperbolic Limits and Convergence Proof\protect\thanks{Dedicated to Abdelghani Bellouquid who prematurely passed away on August 2015.}}

\author{Mohamed Khaladi, Nisrine Outada\thanks{Universit\'e Cadi Ayyad, Facult\'e des Sciences Semlalia, LMDP, UMMISCO (IRD- UPMC), Marrakech 40000, B.P. 2390, Maroc}
and 
Nicolas Vauchelet\thanks{Universit\'e Paris 13, Sorbonne Paris Cit\'e, Laboratoire Analyse G\'eom\'etrie et Applications, CNRS UMR 7539, 93430 Villetaneuse, France}
}


\maketitle

\begin{abstract}
  We consider a system of two kinetic equations modelling a multicellular system : The first equation governs the dynamics of cells, whereas the second kinetic equation governs the dynamics of the chemoattractant.
  For this system, we first prove the existence of global-in-time solution. The proof of existence relies on a fixed point procedure after establishing some a priori estimates. Then, we investigate the hyperbolic limit after rescaling of the kinetic system. It leads to a macroscopic system of Cattaneo type. The rigorous derivation is established thanks to a compactness method.
\end{abstract}


\section{Introduction}
Our paper deals with derivation of models suitable to describe the behavior of multicellular systems from their description at the microscopic scale delivered by models derived by suitable generalizations of the kinetic theory. This problem can be viewed as a possible generalization of the celebrated sixth Hilbert problem~\cite{[H902]} which has been object of several interesting contributions in the classical kinetic theory. The literature in the field is documented in the review papers by  Perthame~\cite{[PE04]} and Saint Raymond~\cite{[S09]}.  As it is known, the time-space scaling can be referred to the so called parabolic and hyperbolic limits or equivalently low and high field limits. The parabolic limit leads to a drift--diffusion type system (or reaction--diffusion system) in which the diffusion processes dominate the behavior of the solutions. The hyperbolic limit leads to models where the influence of the diffusion terms is of lower (or equal) order of magnitude in comparison with other convective or interaction terms. Accordingly, different macroscopic models are obtained corresponding to different scaling assumptions.

The derivation of macroscopic equations from the kinetic theory description was introduced for dispersed biological entities in the pioneer paper~\cite{[ODA88]} and subsequently developed by various authors as witnessed in the bibliography of the survey~\cite{[BBNS12]}. An interesting application has been the derivation of Keller-Segel type models. A broad bibliography has been produced on this challenging topic as reviewed in Sections 5 and 6 of the survey~\cite{[BBTW15]}. The rationale of the approach proposed in~\cite{[ODA88]} consists in deriving a kinetic type model corresponding to the transport equation where the collision operator, namely the right hand side term of the kinetic equation, is perturbed by small stochastic term modeling a poisson velocity jump process. The small parameter corresponds to the entity of the perturbation, while an expansion of the dependent variable is developed in terms of powers of the said parameter. Very recent applications have been delivered in~\cite{[BBC16],[BC17],[OVAK16]}.

This approach is useful even when it developed at a formal level as it leads to interesting models at the macroscopic scale based on models of the dynamics at the microscopic scale rather than on artificial assumptions to close mass and momentum conservation equations. However, as it is known, most of the literature is developed at a formal level, where \textit{ad hoc} assumptions are needed to prove convergence of the aforementioned  power  expansions. The derivation of hyperbolic models involves additional problems on the convergence of Hilbert type expansions technically related to loss of regularity. Indeed, this is the main challenge of our paper which is tackled  in four sections. In more details, Section 2 presents a kinetic theory model of cross diffusion phenomena, where an hyperbolic scaling is is used to include propagation phenomena with finite speed; a binary mixture is accounted for and the statement of the initial value problem is delivered. Section 3 develops a qualitative analysis of the initial value problem and ends up with a local, in time, existence result and with the extension to arbitrarily large times. Finally, a convergence proof of an Hilbert type expansion is delivered in Section 4, however, due to technical difficulties, we restrict ourself to one dimension.

\section{A kinetic model of chemotaxis}
In this section we recall briefly the kinetic model presented in \cite{[OVAK16]}. For this aim, let $f(t,x,v)$ and $g(t,x,v)$ denotes, respectively, the density of cells and of the chemoattractant, depending on time $t$, position $x \in \mathbb{R}^d$ and velocity $v \in V \subseteq \mathbb{R}^d$. Then our kinetic model of chemotaxis reads:
\begin{equation}\label{SYST}
\begin{cases}
\partial_t f + v \cdot \nabla_{x}f= L(g,f),\vspace{0.25cm}\\
\partial_t g + v \cdot \nabla_{x}g= l(g)+G(f,g),
\end{cases}
\end{equation}
where the perturbation turning operators $L$ and $l$ model the dynamics of biological organisms by velocity-jump process, and are integral operators defined by 
\begin{equation}
L(f)=\int_{V}\left(T(v,v')f(t,x,v')-T(v',v)f(t,x,v)\right)dv',
\end{equation}
\begin{equation}\label{operatorl}
l(f)=\int_{V}\left(\mathcal{K}(v,v')f(t,x,v')-\mathcal{K}(v',v)f(t,x,v)\right)dv',
\end{equation}
while the operator $G(f,g)$, which describe proliferation/destruction interactions, is given by
\begin{equation}
G(f,g) = a \left< f \right> - b \left< g \right>, 
\end{equation}
where $a$, $b$ are real positive constants, and $\left<\cdot\right>$ stands for the $(v)$-mean of a function, i.e $\displaystyle{\left<h\right>:=\int_{V} h(t,x,v)dv}$ for $h \in L^1(V)$. The turning kernels $T(v,v')$ and $K(v,v')$ describe the reorientation of cells, i.e the random velocity changes from the previous velocity $v'$ to the new $v$.
Moreover, it is assumed that the set of admissible velocities $V$ is a spherically symmetric bounded domain of $\RR^d$ with $V\subset B_\nu$ (the ball of radius $\nu>0$).
This corresponds to the assumption that any individual of the population chooses a direction with bounded velocity.

As it is mentioned in the introduction, our contribution in this paper will be the rigorous derivation of a diffusive type model for movement of chemotaxis, obtained as a hydrodynamic limit of the kinetic model \eqref{SYST}. In detail, let us assume a hyperbolic scaling for the first population:
\begin{equation}
x \longrightarrow \varepsilon x, \quad t \longrightarrow \varepsilon t,
\end{equation} 
where $\varepsilon > 0$ is a small parameter which will be allowed to tend to zero. In this way we obtain from \eqref{SYST} the following scaled kinetic equation 
\begin{equation}\label{csks}
\begin{cases}
\partial_t f_\varepsilon + v \cdot \nabla_{x}f_\varepsilon= \frac{1}{\varepsilon} L(g_\varepsilon,f_\varepsilon),\vspace{0.25cm}\\
\partial_t g_\varepsilon + v \cdot \nabla_{x}g_\varepsilon= l(g_\varepsilon)+G(f_\varepsilon,g_\varepsilon).
\end{cases}
\end{equation} 
In addition we assume that the operator $L$ admits the following decomposition:
\begin{equation}\label{decom}
L(g_\varepsilon,f_\varepsilon)=L^0(f_\varepsilon)+\varepsilon L^1(g_\varepsilon,f_\varepsilon),
\end{equation}
where the perturbation turning operators $L^0$ and $L^1$ are linear integral operators with respect to $f_\varepsilon$, and reads:
\begin{equation}\label{n0}
L^{0}(f_\varepsilon)=\int_{V}\left(T^{0}(v,v')f_\varepsilon(t,x,v')-T^{0}(v',v)f_\varepsilon(t,x,v)\right)dv',
\end{equation}
\begin{equation}\label{n1}
L^{1}[g_\varepsilon](f_\varepsilon)=\int_{V}\left(T^{1}[g_\varepsilon](v,v')f_\varepsilon(t,x,v')-T^{1}[g](v',v)f_\varepsilon(t,x,v)\right)dv',
\end{equation}
while the operator $l$ is still defined by Eq. \eqref{operatorl}.
In this work we consider the following turning kernels $T^0$, $T^1$, and $\mathcal{K}$ given by 
\begin{equation}\label{T0gamma}
T^{0}(v,v')=\frac{\mu_0}{|V|}(1+ \gamma^2 v \cdot v'), \qquad \gamma^2 \int_V v\otimes v\,dv = |V| I_d,
\end{equation}
\begin{equation}
T^{1}[g](v,v')=\frac{\mu_1}{|V|}-\frac{\mu_{2} \gamma^2}{|V|}v'\cdot \alpha(<g_\varepsilon>),
\end{equation}
\begin{equation}
\mathcal{K}(v,v')=\frac{\sigma}{|V|},
\end{equation}
where $\mu_0$, $\mu_1$, $\mu_2$, $\sigma$ are real positive constants, $\alpha$ is a mapping $ \mathbb{R} \longrightarrow \mathbb{R}^d$, and $\vert V \vert$ denotes the volume of $V$. Notice that since $V$ is assumed to be spherically symmetric, the constant $\gamma$ in \eqref{T0gamma} is well-defined.

With these considerations and after a straightforward calculation we obtain the following kinetic system, we refer to the paper \cite{[OVAK16]} for more details,  
\begin{equation}\label{cks}
\begin{cases}
\displaystyle \partial_t f_\varepsilon+ v\cdot \nabla_x f_\varepsilon = \dfrac{\mu_0}{\varepsilon} \left( F_{J_\varepsilon} - f_\varepsilon \right) + \mu_1 \left( \frac{n_\varepsilon}{\vert V \vert} - f_\varepsilon \right)\\
\hspace{2.8cm} \displaystyle  - \mu_2 \gamma^2 \left( \frac{J_\varepsilon}{\vert V \vert} -v f_\varepsilon \right) \cdot \alpha(S_\varepsilon) , \vspace{0.25cm}\\
\displaystyle \partial_t g_\varepsilon + v \cdot \nabla_x g_\varepsilon = \sigma \left( \frac{S_\varepsilon}{\vert V \vert} -g_\varepsilon \right) + a n_\varepsilon - bS_\varepsilon,
\end{cases}
\end{equation}
where:
\begin{itemize}
\item The local densities $n_\varepsilon(t,x)$ and $S_\varepsilon(t,x)$ are defined by
\begin{equation*}
n_\varepsilon(t,x)=\int_V f_\varepsilon(t,x,v)dv, \quad \text{and} \quad  S_\varepsilon(t,x) = \int_V g_\varepsilon(t,x,v) dv,
\end{equation*}
while the flux function $J_\varepsilon(t,x)$ fulfills
\begin{equation*}
J_\varepsilon(t,x) = \int_V v f_\varepsilon(t,x,v) dv.
\end{equation*}

\item The equilibrium function $F_{J_\varepsilon}(t,x,v)$ is assumed to be a linear combination of $1$, $v_1, \ldots , v_d$:
\begin{equation}\label{defFJ}
F_{J_\varepsilon}(t,x,v) = \dfrac{1}{\vert V \vert} \left( n_\varepsilon(t,x) + \gamma^2 J_\varepsilon(t,x) \cdot v \right).
\end{equation}
This equilibrium function is such that (see \eqref{T0gamma} for the definition of $\gamma$)
$$
\int_V F_{J_\varepsilon}(t,x,v)\,dv = n_\varepsilon(t,x), \qquad
\int_V v F_{J_\varepsilon}(t,x,v)\,dv = J_\varepsilon(t,x).
$$
\end{itemize}
This system is completed with initial condition
\begin{equation}\label{ic}
f_\varepsilon(0,x,v) = f^0_\varepsilon(x,v), \quad \text{and} \quad g_\varepsilon(0,x,v) = g^0_\varepsilon(x,v).
\end{equation}

\section{Existence result}
The existence of solutions to kinetic models of chemotaxis coupled to parabolic or elliptic system for the chemoattractant concentration has been studied in several papers (see for instance \cite{[CMPS04],[EH06],[HKS],[V10]}). However, the study of coupled kinetic systems like Eq. \eqref{cks} is less common.

The aim of this section is to study the Cauchy problem \eqref{cks}-\eqref{ic} for fixed $\varepsilon>0$. More in detail we will state and prove an existence and uniqueness result for the kinetic model \eqref{cks}-\eqref{ic} in Theorem \ref{th:exist}. The proof is based on a fixed point procedure, after establishing some a priori estimates. 

We now introduce some notations which will be used throughout this section: $X_T:=L^\infty((0,T) \times \mathbb{R}^d \times V)$ stands for the Lebesgue space of essentially bounded measurable functions, with norm given by
\begin{equation*}
\left \| f \right \|_{L^\infty_{t,x,v}} = \inf \left\lbrace C \geq 0; \; \vert f(x) \vert \leq C \; \text{for almost every} \; (t,x,v) \in (0,T) \times \mathbb{R}^d \times V \right\rbrace,
\end{equation*}
and we have analogous definitions for $L^\infty_{x}$, $L^\infty_{t,x}$ and $L^\infty_{x,v}$. 
Moreover, we define $X_T^+$ the subspace of $X_T$ with nonnegative functions.

We assume that $\alpha$ is a bounded and globally Lipschitz continuous function on $\mathbb{R}$:
There exists $\alpha_\infty>0$, $L_\alpha>0$ such that
\begin{equation}\label{(H)}
\forall\, S_1,S_2\in\mathbb{R}, \qquad
\| \alpha(S_1)\|\leq \alpha_\infty, \quad
\| \alpha \left( S_1 \right) - \alpha \left( S_2 \right) \| \leq L_\alpha |S_1 - S_2|.
\end{equation}

\begin{definition}
We say that $f$ is a weak solution of \eqref{cks}--\eqref{ic} on $X_T$  for $T>0$, if $f \in X_T$ and satisfies
\begin{equation*}
\begin{cases}
&\displaystyle\int_{(0,T) \times \mathbb{R}^d \times V}\left(\partial_t \varphi + v \cdot \nabla_x \varphi \right) f \,dxdvdt = -\dfrac{\mu_0}{\varepsilon} \int_{(0,T) \times \mathbb{R}^d \times V} \left(F_J -f\right) \varphi \,dxdvdt \\[2mm]
& \qquad\quad \displaystyle - \mu_1 \int_{(0,T) \times \mathbb{R}^d \times V} \left( \dfrac{n}{\vert V \vert}-f \right) \varphi \,dxdvdt  - \int_{\mathbb{R}^d \times V} f^0(x,v) \; \varphi(0,x,v) \,dxdv  \vspace{0.2cm}\\
& \qquad\quad \displaystyle + \mu_2 \gamma^2 \int_{(0,T) \times \mathbb{R}^d \times V} \left( \dfrac{J}{\vert V \vert} -vf \right) \cdot \alpha(S) \,  \varphi \,dxdvdt, \vspace{0.25cm}\\
&\displaystyle\int_{(0,T) \times \mathbb{R}^d \times V} \displaystyle\left(\partial_t \varphi + v \cdot \nabla_x \varphi \right) g \,dxdvdt = -\sigma \int_{(0,T) \times \mathbb{R}^d \times V} \left( \dfrac{n}{\vert V \vert} -g \right) \varphi \,dxdvdt \vspace{0.2cm}\\
& \qquad\quad \displaystyle + \int_{(0,T) \times \mathbb{R}^d \times V} \left( an - bS \right) \varphi \,dxdvdt - \int_{\mathbb{R}^d \times V} g^0(x,v) \; \varphi(0,x,v)\,dxdv,
\end{cases}
\end{equation*}
for any test function $\varphi \in \mathcal{D}([0,T) \times \mathbb{R}^d \times V)$. 
\end{definition}

We now state the main result of this section.
\begin{theorem}[Existence of weak solutions]
\label{theorem1}\label{th:exist}
Let $(f^0,g^0)\in L^\infty_{x,v}\times L^\infty_{x,v}$ be nonnegative and assume that $\alpha$ satisfies assumption \eqref{(H)}. Then the Cauchy problem \eqref{cks}-\eqref{ic} has a unique global weak solution $(f,g)$, with $(f,g)\in X_T^+\times X_T^+$.

Moreover, if $(f^0,g^0)\in L^1_{x,v}\times L^1_{x,v}$, then for any $t\in [0,T]$, 
$\|f(t,\cdot,\cdot)\|_{L^1_{x,v}} = \|f^0\|_{L^1_{x,v}}$ and 
$\|g(t,\cdot,\cdot)\|_{L^1_{x,v}} = \frac ab \|f^0\|_{L^1_{x,v}}(1-e^{-b|V|t}) + \|g^0\|_{L^1_{x,v}} e^{-b|V|t}$.
\end{theorem}

The proof of Theorem \ref{theorem1} is divided into several steps. We first establish some a priori estimates thanks to a characteristics method. Then, applying a fixed point procedure, we establish the existence of a local in time solution. This solution can be extended for arbitrary time $T>0$ and therefore we get a global existence result.

\subsection{A priori estimates}\label{sub1}

We start with the following a priori estimates.
\begin{lemma}[A priori estimates] 
\label{lemma1}
Let $T>0$ and suppose that $\alpha$ satisfies assumption \eqref{(H)}.
Let $(f^0,g^0)$ be given in $L^\infty_{x,v}\times L^\infty_{x,v}$.
Let $(f,g)$ be a weak solution of \eqref{cks}-\eqref{ic} such that $(f,g)\in X_T^+ \times X_T^+$ and $(\nabla_x f, \nabla_x g)\in {X_T}^d \times {X_T}^d$. Then $(f,g)$ satisfies the following estimates:
\begin{equation}\label{lem 2.1}
\|n \|_{L^\infty _{t,x }}+ \|f\|_{X_T} \leq C_1 \|f^0\|_{L^\infty _{x,v }},
\end{equation}
\begin{equation}\label{lem 2.2}
\|S \|_{L^\infty _{t,x }}+\|g\|_{X_T} \leq C_2 \big(\|f^0\|_{L^\infty _{x,v }}+\|g^0\|_{L^\infty _{x,v }}\big).
\end{equation}

Furthermore,  if the initial data $(f^0,g^0)\in L^1_{x,v}\times L^1_{x,v}$ then we have,
$\forall\,t\in [0,T]$, $\|f(t,\cdot,\cdot)\|_{L^1_{x,v}} = \|f^0\|_{L^1_{x,v}}$, and
$$
\|g(t,\cdot,\cdot)\|_{L^1_{x,v}} = \frac ab \|f^0\|_{L^1_{x,v}}(1-e^{-b |V| t}) + \|g^0\|_{L^1_{x,v}} e^{-b|V|t}.
$$

Moreover, if the initial data are given in $W^{1,\infty}_{x,v}\times W^{1,\infty}_{x,v}$ and assuming that $\alpha\in C^1(\RR)$, then
\begin{equation}\label{lem 2.3}
\|\nabla _x n \|_{(L^\infty _{x,v})^d}+\|\nabla_x f\|_{(X_T)^d} \leq C_3 \big(\|\nabla_x f^0\|_{(L^\infty _{x,v })^d}+\|\nabla_x g^0\|_{(L^\infty _{x,v })^d}\big),
\end{equation}
\begin{equation}\label{lem 2.4}
\|\nabla _x S \|_{(L^\infty _{x,v})^d}+\|\nabla_x g\|_{(X_T)^d} \leq C_4 \big(\|\nabla_x f^0\|_{(L^\infty _{x,v })^d}+\|\nabla_x g^0\|_{(L^\infty _{x,v })^d}\big),
\end{equation}
where the constants $C_i,\; i=1,2,3,4,$ are independents of time $T>0$.
\end{lemma}

\begin{proof}
{1.} First we begin with the proof of Eq. (\ref{lem 2.1}). For this purpose we write the first equation of system \eqref{cks} in the following way
\begin{equation}\label{2.5}
\partial_t f + v \cdot \nabla_x f + Kf = R_1,
\end{equation}
where the functions $K$ and $R_1$ are given by
\begin{equation}
K=\dfrac{\mu_0}{\varepsilon} +\mu_1 -\mu_2 \gamma^2 v \cdot \alpha(S), \;\; \text{and} \;\; R_1=\dfrac{\mu_0}{\varepsilon}F_{J}+ \dfrac{\mu_{1} n}{|V|} - \mu_{2}\gamma^2 J \cdot \alpha(S),
\end{equation}
where the expression of $F_J$ is given in \eqref{defFJ}.
Integrating \eqref{2.5} along the characteristics, we get
\begin{equation}\label{2.7}
\begin{split}
f(t,x,v)  =  \exp\left(\int_t^0 K(\tau,\widetilde{x}_\tau, v )d\tau \right) f^0(x-tv,v)\hspace*{1.1cm} \\
+ \int_0^t \exp\left( \int_t^s K(\tau,\widetilde{x}_\tau, v )d\tau \right) R_1(x,\widetilde{x}_s,v)ds,
\end{split}
\end{equation}
where we set $\widetilde{x}_\tau=x+(\tau - t)v$ (this notation will be used throughout this section).
Moreover,  using assumption \eqref{(H)}, for each $0\leq s \leq \tau \leq t\leq T $ we have
\begin{equation}
\left|K(\tau , \widetilde{x}_\tau, v )\right| \leq \dfrac{\mu_0}{\varepsilon}+\mu_1+\mu_2\alpha_\infty\gamma^2\nu.
\end{equation}
It follows
\begin{equation}\label{2.8}
\exp\left( \int_t^s K(\tau,\widetilde{x}_\tau, v )d\tau \right)\leq e^{C_1 T} \leq C_2.
\end{equation}
According to Eqs. (\ref{2.7}) and (\ref{2.8}) we write
\begin{equation}\label{2.9}
f(t,x,v)\leq C_2 f^0(x-tv,v)+C_2 \int_0^t |R_1(s,\widetilde{x}_s, v)|ds.
\end{equation}
We estimate the last term of the right hand side of the later inequality as follows:
\begin{equation*}
\begin{split}
\int_0^t |R_1(s,\widetilde{x}_s,v)|ds
\leq &\left(\dfrac{\mu_0}{\varepsilon|V|}+\dfrac{\mu_1}{|V|}\right)\int_0^tn(s,\widetilde{x}_s)ds \\
&+\left(\dfrac{\mu_0\gamma^2\nu}{\varepsilon |V|} + \mu_2 \alpha_\infty\gamma^2 \right)\int_0^t|J(s,\widetilde{x}_s)|ds\\
 \leq & \left[\dfrac{\mu_0}{\varepsilon|V|}+\dfrac{\mu_1}{|V|}+\dfrac{\mu_0\gamma^2\nu^2}{\varepsilon |V|}+\mu_2 \alpha_\infty\gamma^2\nu\right] \int_0^t n(s,\widetilde{x}_s)\,ds\\
 \leq & C_3 \int_0^t \|n(s,.)\|_{L^\infty_x}\,ds.
\end{split}
\end{equation*}
Injecting this last estimate in (\ref{2.9}), we obtain
\begin{equation}\label{2.10}
f(t,x,v)\leq C_2 \|f^0\|_{L^\infty _{x,v}}+C_4\int_0^t \| n(s, \cdot) \|_{L^\infty _x}\,ds.
\end{equation}
An integration with respect to $v$ provides
\begin{equation}
\|n(t,\cdot)\|_{L^\infty _x}\leq C_2|V|\; \|f^0\|_{L^\infty _{x,v}}+C_4 |V| \int_0^t \| n(s,\cdot) \|_{L^\infty _x}\,ds.
\end{equation}
Therefore, applying Gronwall's inequality we get
\begin{equation}\label{2.11}
\| n(t, \cdot) \|_{L^\infty _x}\leq C \|f^0\|_{L^\infty _{x,v}}.
\end{equation}

Using Eq. (\ref{2.10}) together with (\ref{2.11}), we obtain a similar bound on $f$ in $L^\infty_{x,v}$.
This completes the proof of the first assertion (\ref{lem 2.1}).

{2.} The proof of (\ref{lem 2.2}) is straightforward and follows the same ideas as of estimate \eqref{lem 2.1}. Indeed, we have
\begin{equation}\label{2.12}
\partial_t g + v \cdot \nabla_x g + \sigma g = R_2, \;\; \text{where} \;\; R_2= \left(\dfrac{\sigma}{|V|}-b\right)S+a\,n.
\end{equation}
Integrating along the characteristics, we get
\begin{equation}
g(t,x,v)=e^{-\sigma t}g^0(x-tv,v)+\int_0^te^{(s-t)\sigma}R_2(s,\widetilde{x}_s)ds,
\end{equation}
and easy computation yields
\begin{eqnarray}\label{12.5}
g(t,x,v) & \leq & g^0(x-tv,v)+\int_0^t|R_2(s,\widetilde{x}_s)|ds\nonumber\\
& \leq & \|g^0\|_{L^\infty _{x,v}}+\left| \dfrac{\sigma}{|V|}-b\right|\int_0^t|S(s,\widetilde{x}_s)|ds+ a \int_0^t|n(s,\widetilde{x}_s)|ds \nonumber\\
& \leq & \|g^0\|_{L^\infty _{x,v}}+\left| \dfrac{\sigma}{|V|}-b\right|\int_0^t\|S(s, \cdot)\|_{L^\infty _x}\,ds
+ a \int_0^t \| n(s,\cdot) \|_{L^\infty _x}\,ds. \qquad
\end{eqnarray}
According to (\ref{lem 2.1}) we can write $$\|n(s,.)\|_{L^\infty _x}\leq C_1\|f^0\|_{L^\infty _{x,v}},$$
hence, from \eqref{12.5} it follows that
\begin{equation}\label{2.13}
g(t,x,v)\leq \|g^0\|_{L^\infty _{x,v}}+C_1\|f^0\|_{L^\infty _{x,v}}+C_2 \int_0^t\|S(s, \cdot)\|_{L^\infty_x}\,ds.
\end{equation}
Integrating over $V$, we obtain
\begin{equation}
S(t,x)\leq |V|\, \|g^0\|_{L^\infty _{x,v}} + C_1 |V| \,\|f^0\|_{L^\infty _{x,v}}+C_2|V| \int_0^t\|S(s,\cdot)\|_{L^\infty _x}\,ds,
\end{equation}
and we estimate $S$ thanks to Gronwall's inequality and we conclude the proof of (\ref{lem 2.2}) with (\ref{2.13}).

{3.} Assuming the initial data in $L^1_{x,v}$, we have by integration of the first equation in \eqref{cks}: $\|f(t,\cdot,\cdot)\|_{L^1_{x,v}} = \|f^0\|_{L^1_{x,v}}.$
Integrating the second equation in \eqref{cks}, we get
$$
\frac{d}{dt}\|g(t,\cdot,\cdot)\|_{L^1_{x,v}} = a |V| \|f^0\|_{L^1_{x,v}} - b |V| \|g(t,\cdot,\cdot)\|_{L^1_{x,v}}.
$$
We obtain the desired estimate by integrating in time this later identity.

{4.} We now prove (\ref{lem 2.3}) and (\ref{lem 2.4}). To begin with, we rewrite \eqref{cks} in the following way
\begin{equation}
\begin{cases}
\partial_t f + v \cdot \nabla_x f + \widetilde{K}f = \widetilde{R}_1,\\
\partial_t g + v \cdot \nabla_x g  = R_2,
\end{cases}
\end{equation}
where the functions $\widetilde{K}$ and $\widetilde{R}_1$ are defined by
\begin{equation}
\widetilde{K}  =  \dfrac{\mu_0}{\varepsilon}+\mu_1, \;\; \text{and}\;\; \widetilde{R}_1  =  \dfrac{\mu_0}{\varepsilon}F_{J}+\dfrac{\mu_1\,n}{|V|} - \frac{\mu_2\gamma^2}{|V|} J \cdot \alpha(S)+\mu_2\gamma^2 v \cdot \alpha(S)f,
\end{equation}
while $R_2$ is still given in (\ref{2.12}). Therefore, we obtain
\begin{equation} \label{2.14}
f(t,x,v)=e^{-t\widetilde{K}}f^0(x-tv,v)+\int_0^te^{(s-t)\widetilde{K}}\widetilde{R}_1(s,\widetilde{x}_s,v)ds,
\end{equation}
and
\begin{equation}\label{2.15}
g(t,x,v)=e^{-t\sigma}g^0(x-tv,v)+\int_0^te^{(s-t)\sigma}R_2(s,\widetilde{x}_s)ds.
\end{equation}
Let $i\in\{1,\ldots,d\}$ be arbitrary but fixed index, and for a generic function $h$ we denote by $h_i$ the partial derivate $\partial_{x_i}h$. Hence, from (\ref{2.14}) and (\ref{2.15}) we get
\begin{equation} \label{2.16}
f_i(t,x,v)=e^{-t\widetilde{K}}f_i^0(x-tv,v)+\int_0^te^{(s-t)\widetilde{K}}\partial_{x_i}\left(\widetilde{R}_1(s,\widetilde{x}_s,v)\right) ds,
\end{equation}
and
\begin{equation}\label{2.17}
g_i(t,x,v)=e^{-t\sigma}g_i^0(x-tv,v)+\int_0^te^{(s-t)\sigma}\partial_{x_i}\left(R_2(s,\widetilde{x}_s)\right) ds.
\end{equation}
We now estimate separately $f_i$ and $g_i$. From (\ref{2.16}) it follows that
\begin{equation} \label{2.18}
|f_i(t,x,v)| \leq \|f^0_i\|_{L^\infty_{x,v}}+\int_0^t \left|\partial_{x_i}\left(\widetilde{R}_1(s,\widetilde{x}_s,v)\right)\right| ds.
\end{equation}
We have
\begin{equation}\label{2.19}
\begin{split}
\partial_{x_i} & \big( \widetilde{R}_1(s, \widetilde{x}_s,v) \big)  =  \dfrac{\mu_0}{\varepsilon |V|}\left(n_i(s,\widetilde{x}_s)+\gamma^2 J_i(s,\widetilde{x}_s) \cdot v\right)+\dfrac{\mu_1\,n_i(s,\widetilde{x}_s)}{|V|}\\
&-\frac{\mu_2\gamma^2}{|V|} J_i(s,\widetilde{x}_s) \cdot \alpha\left(S(s,\widetilde{x}_s)\right) - \frac{\mu_2\gamma^2}{|V|} S_i(s,\widetilde{x}_s)J(s,\widetilde{x}_s) \cdot \alpha ' \left(S(s,\widetilde{x}_s)\right)\\
&+\mu_2\gamma^2 v \cdot \alpha\left(S(s,\widetilde{x}_s)\right)\,f_i(s,\widetilde{x}_s,v) +\mu_2\gamma^2 S_i(s,\widetilde{x}_s)\,v \cdot \alpha ' \left(S(s,\widetilde{x}_s)\right)\,f(s,\widetilde{x}_s,v) \\
= &  \left( \dfrac{\mu_0}{\varepsilon |V|}+ \dfrac{\mu_1}{|V|}\right)n_i(s,\widetilde{x}_s)+\left( \dfrac{\mu_0\gamma^2\,v}{\varepsilon |V|}- \frac{\mu_2\gamma^2}{|V|} \alpha\left(S(s,\widetilde{x}_s)\right)\right) \cdot J_i(s,\widetilde{x}_s)\\
& -\frac{\mu_2\gamma^2}{|V|} S_i(s,\widetilde{x}_s)J(s,\widetilde{x}_s) \cdot \alpha ' \left(S(s,\widetilde{x}_s)\right) + \mu_2\gamma^2 v \cdot \alpha\left(S(s,\widetilde{x}_s)\right)\,f_i(s,\widetilde{x}_s,v) \\
& + \mu_2\gamma^2 S_i(s,\widetilde{x}_s)\,v \cdot \alpha' \left(S(s,\widetilde{x}_s)\right)\,f(s,\widetilde{x}_s,v).
\end{split}
\end{equation}
We introduce the following notations
\begin{equation}
\widetilde{n}_i(s)=\int_V \|f_i(s,\cdot,v)\|_{L^\infty_x} dv, \text{\quad and \quad} \widetilde{S}_i(s)=\int_V \|g_i(s,\cdot,v)\|_{L^\infty_x} dv.
\end{equation}
In this way we have
\begin{equation}
|n_i(s,\widetilde{x}_s)| \leq \widetilde{n}_i(s), \quad |J_i(s,\widetilde{x}_s)|\leq \nu \widetilde{n}_i(s), \quad \text{and} \quad |S_i(s,\widetilde{x}_s)|\leq \widetilde{S}_i(s).
\end{equation}
Then from (\ref{2.19}) we immediately obtain
\begin{equation}\label{2.20}
\begin{split}
\Big | \partial_{x_i}\big(& \widetilde{R}_1(s,\widetilde{x}_s,v)\big) \Big|  \leq  C_1 \widetilde{n}_i(s) +C_2 \widetilde{S}_i(s)\|n(s,\cdot)\|_{L^\infty_x}\\
& +C_3 \|f_i(s,\cdot,v)\|_{L^\infty_{x}} +C_4\widetilde{S}_i(s)\|f(s,\cdot,\cdot)\|_{L^\infty_{x,v}}.
\end{split}
\end{equation}
According to (\ref{lem 2.1}) we have
\begin{equation}
\|n(s,\cdot)\|_{L^\infty_x} \leq C_1 \|f^0\|_{L^\infty_{x,v}}, \text{\quad and \quad} \|f(s,\cdot,\cdot)\|_{L^\infty_{x,v}}\leq C_2\|f^0\|_{L^\infty_{x,v}}.
\end{equation}
Therefore, using (\ref{2.20}) we deduce that
\begin{equation}
\Big| \partial_{x_i}\big( \widetilde{R}_1(s,\widetilde{x}_s,v) \big) \Big|  \leq  C_1 \widetilde{n}_i(s) +C_2 \widetilde{S}_i(s)+C_3 \|f_i(s,\cdot,v)\|_{L^\infty_{x}}.
\end{equation}
This last estimate together with (\ref{2.18}) allow to write
\begin{equation}\label{2.21}
 \|f_i(t,\cdot,v)\|_{L^\infty_x} \leq \|f_i^0 \|_{L^\infty_{x,v}} + C_1 \int_0^t \left(\widetilde{n}_i(s) + \widetilde{S}_i(s)\right) ds + C_3 \int_0^t \|f_i(s,\cdot,v)\|_{L^\infty_x} ds.
\end{equation}

The estimate on $g_i$ can be done similarly to assertion (\ref{2.21}). Indeed from Eq. (\ref{2.17}) it follows that
\begin{equation}\label{2.22}
 \|g_i(t,\cdot,v)\|_{L^\infty_x}\leq \|g_i^0\|_{L^\infty_{x, v}} +  \int_0^t \left|\partial_{x_i}\left(R_2(s,\widetilde{x}_s)\right)\right| ds,
\end{equation}
and we compute the first partial derivative of $R_2$ as follows
\begin{equation}
\partial_{x_i}\left(R_2(s,\widetilde{x}_s)\right)=\left( \dfrac{\sigma}{|V|}-b \right)S_i+n_i.
\end{equation}
Hence
\begin{equation}\label{2.23}
\partial_{x_i}\left(R_2(s,\widetilde{x}_s)\right)\leq \left| \dfrac{\sigma}{|V|}-b \right|\widetilde{S}_i(s)+\widetilde{n}_i(s).
\end{equation}
Taking Eqs. (\ref{2.22}) and (\ref{2.23}) into account we deduce that
\begin{equation}\label{2.24}
  \|g_i(t,\cdot,v)\|_{L^\infty_x}\leq \|g_i^0\|_{L^\infty_{x,v}} +  \int_0^t \left( \widetilde{n}_i(s)+\widetilde{S}_i(s) \right) ds.
\end{equation}
Next integrating, with respect to $v$. Eqs. (\ref{2.21}) and (\ref{2.24}) and adding the resulting inequalities, we can write
\begin{equation}\label{24.5}
\widetilde{n}_i(t)+\widetilde{S}_i(t)\leq C_1 \left( \|f_i^0\|_{L^\infty_{x,v}}+\|g_i^0\|_{L^\infty_{x,v}}\right)+ C_2  \int_0^t \left( \widetilde{n}_i(s)+\widetilde{S}_i(s) \right) ds.
\end{equation}
Therefore, in view of  Gronwall's inequality, equation \eqref{24.5} yields
\begin{equation}\label{2.25}
|n_i(s,x)|+|S_i(s,x)|\leq \widetilde{n}_i(s)+\widetilde{S}_i(s)\leq C_1 \left( \|f_i^0\|_{L^\infty_{x,v}}+\|g_i^0\|_{L^\infty_{x,v}}\right),
\end{equation}
and a similar estimate is obtained for $f_i$ and $g_i$ using (\ref{2.21}), (\ref{2.24}) and (\ref{2.25}). This complete the a-priori estimates.
\end{proof}

\subsection{Proof of Theorem \ref{theorem1}.}
We are now in position to prove the existence result.
The idea of the proof follows standard techniques consisting in, first, proving local in time existence by a fixed point procedure, second, iterating this process to obtain global in time existence.\\

For the local in time existence, let $T>0$, we introduce the map
$$\mathcal{F}:X_T \longrightarrow X_T,\qquad f \longmapsto \mathcal{F}(f):= \mathcal{F}_2(\mathcal{F}_1(f))$$
where $G=\mathcal{F}_1(f)$ is a weak solution of the following problem:
\begin{equation*}
\begin{cases}
\displaystyle \partial_tG+v \cdot \nabla_x G=\left(\dfrac{\sigma}{|V|}-b\right) \int_V Gdv+ an -\sigma G,\\
G(0,x,v)=g^0(x,v)\in L^\infty_{x,v},
\end{cases}
\end{equation*}
with the notation $n(t,x)=\int_Vf(t,x,v)dv$, while the functional $\mathcal{F}_2$ is defined by: $F=\mathcal{F}_2(g)$ is a weak solution of
\begin{equation*}
\begin{cases}
\displaystyle \partial_t F+v\cdot \nabla_x F=\dfrac{\mu_0}{\varepsilon}\left[\dfrac{1}{|V|}\left(\int_V F dv +\gamma^2\int_V vFdv \cdot v \right)-F\right] \\
\hspace*{2.6cm}\displaystyle +\mu_1\left(  \dfrac{1}{|V|}\int_V F dv -  F \right) - \mu_2 \gamma^2 \left(\dfrac{1}{|V|} \int_V vFdv - vF\right) \cdot \alpha(S),\\
F(0,x,v)=f^0(x,v)\in L^\infty_{x,v},
\end{cases}
\end{equation*}
with $S(t,x)=\int_Vg(t,x,v)dv$.
Existence of solutions for these two linear systems is now standard.
It is clear, adapting the techniques of Lemma \ref{lemma1} that $\mathcal{F}_1$ and $\mathcal{F}_2$ map $X_T$ into itself. 
Our objective is to show that $\mathcal{F}$ defines a contraction on $X_T$ for $T$ small enough. Let $f_1$ and $f_2$ be given in $X_T$, then we have the following result:

\begin{lemma}\label{lemma2}
For $T>0$ small enough, there exists a constant $C_1(T)<1$ such that
\begin{equation}\label{2.34}
\left\|\mathcal{F}_1(f_1)-\mathcal{F}_1(f_2)\right\|_{X_T}\leq C_1(T)\|f_1-f_2\|_{X_T}.
\end{equation}
\end{lemma}
\begin{proof} We set $G_{12}=\mathcal{F}_1(f_1)-\mathcal{F}_1(f_2)$, then we have
\begin{equation}\label{2.35}
\partial_tG_{12}+v \cdot \nabla_x G_{12}=\left(\dfrac{\sigma}{|V|}-b\right) \int_V G_{12}dv-\sigma G_{12}+a(n_1-n_2),
\end{equation}
with the notations $n_i(t,x)=\int_Vf_i(t,x,v)dv,\; i=1,2$. Analogously to the proof of Lemma \ref{lemma1}, we write identity (\ref{2.35}) in the following way
\begin{equation}\label{2.36}
\partial_tG_{12}+v \cdot \nabla_x G_{12}+\sigma G_{12}=R_1,
\end{equation}
where
\begin{equation}
R_1=\left(\dfrac{\sigma}{|V|}-b\right)\int_V G_{12}dv+a(n_1-n_2).
\end{equation}
Moreover, from equation
\begin{equation}\label{2.37}
\dfrac{d}{ds}\left(e^{(s-t)\sigma}G_{12}(s,\widetilde{x}_s,v)\right)=e^{(s-t)\sigma}R_1(s,\widetilde{x}_s),
\end{equation}
it follows that
\begin{equation}\label{2.38}
G_{12}(t,x,v)=\int_0^t e^{(s-t)\sigma}R_1(s,\widetilde{x}_s) ds.
\end{equation}
$\big($We recall the notation $\widetilde{x}_s=x+(s-t)v$ $\big)$. Since $e^{(s-t)\sigma}<1$, for all $0\leq s\leq t\leq T$ we deduce from (\ref{2.38}) the following estimate
\begin{equation}\label{2.39}
|G_{12}(t,x,v)|\leq \int_0^t |R_1(s,\widetilde{x}_s)| ds.
\end{equation}
However, we have
\begin{equation}
\begin{split}
\left|R_1(s,\widetilde{x}_s)\right| & \leq  \left|\dfrac{\sigma}{|V|}-b \right|\,|V|\,\|G_{12}(s,\cdot,\cdot)\|_{L^\infty_{x,v}}+ a\,|V|\,\|f_1-f_2\|_{X_T}\\
& =  C_1 \,\|G_{12}(s,\cdot,\cdot)\|_{L^\infty_{x,v}} + C_2 \,\|f_1-f_2\|_{X_T}.
\end{split}
\end{equation}
Using this last inequality in Eq. (\ref{2.39}) we get
\begin{equation}
|G_{12}(t,x,v)|\leq \int_0^t C_1 \,\|G_{12}(s,\cdot,\cdot)\|_{L^\infty_{x,v}}  ds + C_2 \,\|f_1-f_2\|_{X_T},
\end{equation}
and the Gronwall lemma gives the desired estimate (\ref{2.34}), which finished the proof of Lemma \ref{lemma2}.
\end{proof}
Now, let us introduce $g_1=\mathcal{F}_1(f_1)$ and $g_2=\mathcal{F}_1(f_2)$. Then we claim that:
\begin{lemma}\label{lemma3}
For $T>0$ small enough, there exists a constant $C_2(T)<1$ such that
\begin{equation}\label{2.40}
\left\|\mathcal{F}_2(g_1)-\mathcal{F}_2(g_2)\right\|_{X_T}\leq C_2(T)\|g_1-g_2\|_{X_T}.
\end{equation}
\end{lemma}
\begin{proof} The proof of Lemma \ref{lemma3} follows the same techniques as in the proof of Lemma \ref{lemma2}, but with more technical difficulties. To begin with we set $F_{12}=\mathcal{F}_2(g_1)-\mathcal{F}_2(g_2)$, then we have
\begin{eqnarray}\label{2.41}
&& \partial_tF_{12}+v \cdot \nabla_x F_{12}  =  \dfrac{\mu_0+\varepsilon\mu_1}{\varepsilon |V|}\int_V F_{12}dv+\dfrac{\mu_0 \gamma^2}{\varepsilon |V|}\int_V vF_{12}dv \cdot v \nonumber\\
& & -\dfrac{\mu_2 \gamma^2}{|V|}\int_V vF_{12}dv \cdot \alpha(S_1)  -\dfrac{\mu_2 \gamma^2}{|V|}\int_V \mathcal{F}_2(g_2)dv \cdot \left(\alpha(S_1)-\alpha(S_2)\right)\\
& &-\left(\dfrac{\mu_0}{\varepsilon}+\dfrac{\mu_1}{|V|}\right)F_{12} + \mu_2 \gamma^2 vF_{12} \cdot \alpha(S_1)+ \mu_2 \gamma^2 v\mathcal{F}_2(g_2) \cdot \left(\alpha(S_1)-\alpha(S_2)\right),\nonumber
\end{eqnarray}
with $S_i(t,x)=\int_Vg_i(t,x,v)dv,\; i=1,2$. We introduce the following notations
\begin{equation*}
K=\dfrac{\mu_0}{\varepsilon}+\dfrac{\mu_1}{|V|}-\mu_2 \gamma^2 v \cdot \alpha(S_1),
\end{equation*}
and
\begin{eqnarray*}
R_1 & = & \dfrac{\mu_0+\varepsilon \mu_1}{\varepsilon |V|}\int_V F_{12} dv +\dfrac{\mu_0 \gamma^2}{\varepsilon|V|}  \int_V vF_{12} dv \cdot v -\dfrac{\mu_2 \gamma^2}{|V|}  \int_V vF_{12} dv \cdot \alpha(S_1)\\
& &  -\dfrac{\mu_2 \gamma^2}{|V|}\int_V \mathcal{F}_2(g_2)dv \cdot \left(\alpha(S_1) -\alpha(S_2)\right)+ \mu_2 \gamma^2 v \mathcal{F}_2(g_2) \cdot \left(\alpha(S_1)-\alpha(S_2)\right).
\end{eqnarray*}
In this way we can write identity (\ref{2.41}) as
\begin{equation}\label{2.42}
\partial_tF_{12}+v \cdot \nabla_x F_{12}+ K F_{12} = R_1.
\end{equation}
A simple calculation shows that
\begin{equation}\label{2.43}
F_{12}(t,x,v)=\int_0^t \left[\exp\left(\int_t^s K(\tau,\widetilde{x}_\tau, v)d\tau  \right)R_1(s,\widetilde{x}_s,v)   \right]ds,
\end{equation}
and in view of estimate $\exp\left(\int_t^s K(\tau,\widetilde{x}_\tau, v)d\tau\right) \leq e^{C_1 T}$, we deduce from (\ref{2.43}) that
\begin{equation}\label{2.44}
|F_{12}(t,x,v)|\leq e^{C_1 T}\int_0^t|R_1(s,\widetilde{x}_s,v)| ds.
\end{equation}
Moreover, it is easy to see that
\begin{equation}\label{2.45}
\begin{split}
|R_1(s,\widetilde{x}_s,v)|  \leq & C_2 n_{12}(s,\widetilde{x}_s)+C_3 n_2(s,\widetilde{x}_s) \left| \alpha(S_1(s,\widetilde{x}_s))-\alpha(S_2(s,\widetilde{x}_s)) \right|  \nonumber\\
&  +C_4 \mathcal{F}_2(g_2)(s,\widetilde{x}_s,v)\left| \alpha(S_1(s,\widetilde{x}_s))-\alpha(S_2(s,\widetilde{x}_s)) \right|,
\end{split}
\end{equation}
with the notation $n_{12}(t,x)=\int_V F_{12}(t,x,v) dv$ and $n_2(t,x)=\int_V\mathcal{F}_2(g_2)(t,x,v) dv$. Using Lemma \ref{lemma1} together with the assumption \eqref{(H)}, we get
 \begin{eqnarray}\label{2.46}
|R_1(s,\widetilde{x}_s,v)| & \leq & C_2 |n_{12}(s,\widetilde{x}_x)|+C_5 \|f^0\|_{L^\infty_{x,v}}\, L_\alpha \| S_1-S_2\|_{L^\infty_{t,x}} 
\end{eqnarray}
We remark that
\begin{equation}\label{2.47}
|n_{12}(s,\widetilde{x}_x)|\leq |V| \|F_{12}(s,\cdot,\cdot)\|_{L^\infty_{x,v}},
\end{equation}
and
\begin{equation}\label{2.48}
\| S_1 - S_2 \|_{L^\infty_{t,x}} \leq |V| \|g_1-g_2\|_{X_T}.
\end{equation}
Then from (\ref{2.44}), (\ref{2.46}), (\ref{2.47}) and (\ref{2.48}) it follows that
\begin{equation}\label{2.49}
\|F_{12}(t,\cdot,\cdot)\|_{L^\infty_{x,v}}\leq e^{C_1 T}\int_0^t \|F_{12}(s,\cdot,\cdot)\|_{L^\infty_{x,v}} ds+TC_3 e^{C_1 T}\|g_1-g_2\|_{X_T},
\end{equation}
and we conclude the proof of Lemma \ref{lemma3} using Gronwall inequality.
\end{proof}
The local existence in Theorem \ref{theorem1} follows from a direct application of the Banach fixed point theorem since $\mathcal{F}$ is a contraction on $X_T$ for $T$ small enough.
This gives existence of a unique solution on $[0,T]$ for small enough $T$.
Thanks to a priori estimates in Lemma \ref{lemma1} we may iterate this process to extend the solution on $[T,2T]$, then on $[2T,3T]$, ... It concludes the proof of Theorem \ref{theorem1}.
\qed

\section{Hyperbolic limit}\label{sec:hyp}
Derivation of macroscopic model from the underlaying description at the microscopic scale, provided by the kinetic theory of active particles, is the subject of a growing literature. In \cite{[CMPS04],Hwang05,Hwang06,[BBNS12],Si14,Liao15} it has been proved that the Keller-Segel \cite{[BBTW15]} model can be derived as the limit of a kinetic model by using a moment method.
The hyperbolic limit is considered in \cite{[FLP05],[BBNS07],NoDEA} leading to the same kind of macroscopic model with small diffusion. More recently these results have been extended in \cite{[OVAK16]} dealing with the coupled kinetic system \eqref{cks}. As a consequence a formal derivation of a class of hyperbolic equations of Cattaneo type is obtained.
The aim of this section is to purpose a rigorous proof of the formal derivation of the hyperbolic limit performed in \cite{[OVAK16]}.
However, due to technical difficulties, we restrict ourself to the one dimensional case, $d=1$. 

The main result can be stated as follows.
\begin{theorem}\label{HLRC}
  Let $T>0$, $d=1$, and $V$ a symmetric bounded domain of $\R$ with $\gamma^2 = |V|\left(\int_V v^2\,dv\right)^{-1}$.
  Let $(f^0,g^0)\in (L^1_{x,v}\cap L^\infty_{x,v})^2$ be nonnegative and assume that $\alpha\in C^1(\mathbb{R})$ satisfies \eqref{(H)}. 
Let $(f_\varepsilon,g_\varepsilon) $ be the unique nonnegative weak solution of the scaled Cauchy problem \eqref{cks} on $[0,T]$. Then there exists a subsequence, denoted in the same way, and a couple $(f,g)$ such that
\begin{equation}
 f_\varepsilon \rightharpoonup f, \quad  g_\varepsilon \rightharpoonup g \quad \text{in}\  L^2_{t, x,v}.
\end{equation}
In addition, the moments
\begin{equation}
n=\int_V f(v) \,dv, \quad  S=\int_V g(v) \,dv, \quad J=\int_V vf(v) \,dv,
\end{equation}
satisfy the following macroscopic system
\begin{equation}\label{systlim}
\begin{cases}
\partial_t n + \partial_x J =0\vspace{0.25cm}\\
\partial_t J + \frac{1}{\gamma^2} \partial_x n = -\mu_1 J + \mu_2 n \, \alpha(S)\vspace{0.25cm}\\
\partial_t g + v\partial_x g = \sigma \left(\frac{S}{|V|}-g\right) + a n - b S.
\end{cases}
\end{equation}
Moreover, the asymptotic limit $f$ satisfies
\begin{equation}
f=\frac{1}{|V|} \left(n+ \gamma^2 J v\right).
\end{equation}
\end{theorem}
The first two equations in system \eqref{systlim} form the so-called Cattaneo system for chemosensitive movement \cite{Dolak,Hillen}. 
Hence a direct consequence of this Theorem (and Theorem \ref{theorem1}) is the existence of a solution for the one dimensional Cattaneo system.

Since the last equation has not been rescaled, it cannot be rewritten as a closed system with macroscopic variable. However, we deduce from the last equation in \eqref{systlim} that the moments $S=\langle g \rangle$ and $q=\langle v g \rangle$ verify the (non-closed) system
\begin{align*}
\partial_t S + \partial_x q =an-b S, \qquad\
\partial_t q + \partial_x Q(g) = -\sigma q,
\end{align*}
where the second order moment $Q$ is defined by
$Q(g) = \int_V v^2 g(v) dv.$

\subsection{Uniform a priori estimates}

We start with the following a priori estimates uniform with respect to $\varepsilon>0$:
\begin{lemma}[A priori estimate in $L^2_{x,v}$]\label{L2-a-priori estimate}
We suppose that we are in the conditions of theorem \ref{HLRC}. Then the following estimate
\begin{equation}
 \| f_\varepsilon(t)\|^2_{L^2_{x,v}}+ \|g_\varepsilon(t)\|^2_{L^2_{x,v}} \leq C(T) \left( \|f^0\|^2_{L^2_{x,v}}+ \|g^0\|^2_{L^2_{x,v}} \right),
\end{equation}
holds true for a.e $t \in (0,T)$, where the constant $C(T)$ is independent of $\varepsilon$.
\end{lemma}

\begin{proof}
We multiply the first equation of system \eqref{cks} by $f_\varepsilon$
\begin{eqnarray*}
\dfrac{1}{2}\left(\partial_t f^2_\varepsilon + v\partial_x f^2_\varepsilon \right) &=&  \dfrac{\mu_0}{\varepsilon}\left[ \dfrac{1}{|V|}\left(n_\varepsilon f_\varepsilon + J_\varepsilon\gamma^2 v f_\varepsilon \right)-f^2_\varepsilon \right]+\mu_1 \left(\dfrac{n_\varepsilon}{|V|}f_\varepsilon-f^2_\varepsilon\right)\\
&&- \mu_2 \gamma^2\left(\dfrac{J_\varepsilon f_\varepsilon}{|V|}-vf^2_\varepsilon \right)\alpha(S_\varepsilon),
\end{eqnarray*}
and integrate over $V$ to obtain
\begin{eqnarray}\label{Lemma16-Eq1}
\displaystyle \dfrac{1}{2}\left(\partial_t \int_V f^2_\varepsilon dv + \partial_x \int_V v f^2_\varepsilon  dv\right) =  \dfrac{\mu_0}{\varepsilon}\left[ \dfrac{1}{|V|}\left(n^2_\varepsilon  + J^2_\varepsilon \gamma^2 \right)-\int_V f^2_\varepsilon dv \right] &&\nonumber\\
\displaystyle  +\mu_1 \left(\dfrac{n^2_\varepsilon}{|V|}-\int_V f^2_\varepsilon dv \right)- \mu_2 \gamma^2\left(\dfrac{J_\varepsilon n_\varepsilon}{|V|}-\int_V vf^2_\varepsilon dv \right) \alpha(S_\varepsilon).&&
\end{eqnarray}
Let us introduce the symmetric and the anti-symmetric part of $f_\varepsilon$ as follows
\begin{eqnarray*}
f_\varepsilon^S(v)& = & \dfrac{1}{2}\left(f_\varepsilon(v) + f_\varepsilon(-v)\right), \quad v\in V,  \\
f_\varepsilon^A(v)& = & \dfrac{1}{2}\left(f_\varepsilon(v) - f_\varepsilon(-v)\right), \quad v\in V.
\end{eqnarray*}
Since $V$ is symmetric, it follows that
\begin{equation}\label{Lamma16-Eq2}
f_\varepsilon = f_\varepsilon^S + f_\varepsilon^A, \quad  n_\varepsilon=\int_V f^S_\varepsilon dv, \quad J_\varepsilon = \int_V v f^A_\varepsilon dv,
\end{equation}
and
\begin{equation}\label{Lemma16-Eq3}
\int_V f_\varepsilon^2 dv  = \int_V \left(f^S_\varepsilon\right)^2 dv+ \int_V \left(f^A_\varepsilon\right)^2 dv.
\end{equation}
Using \eqref{Lemma16-Eq1}-\eqref{Lemma16-Eq3}, we have
\begin{equation}\label{Lemma16-Eq4}
\begin{split}
\dfrac{1}{2}\Bigg(\partial_t \int_V f^2_\varepsilon dv + \partial_x & \int_V v f^2_\varepsilon  dv\Bigg) = \dfrac{\mu_0}{\varepsilon}\Bigg[ \dfrac{1}{|V|}\left( \int_V f^S_\varepsilon dv  \right)^2 -\int_V \left(f^S_\varepsilon\right)^2 dv\\
&+\dfrac{\gamma^2}{|V|}\left( \int_V v f^A_\varepsilon dv  \right)^2 -\int_V \left(f^A_\varepsilon\right)^2 dv \Bigg]\\
&+\mu_1 \left[\dfrac{n^2_\varepsilon}{|V|}-\int_V f^2_\varepsilon dv \right] - \mu_2 \gamma^2\left(\dfrac{J_\varepsilon n_\varepsilon}{|V|}-\int_V vf^2_\varepsilon dv \right)\alpha(S_\varepsilon),
\end{split}
\end{equation}
and according to Cauchy-Schwarz inequality we have
\begin{equation}\label{Lemma16-Eq5}
\left( \int_V f^S_\varepsilon dv  \right)^2  \leq  |V| \int_V \left(f^S_\varepsilon\right)^2 dv, \quad
\left( \int_V v f^A_\varepsilon dv  \right)^2 \leq \frac{|V|}{\gamma^2} \int_V \left(f^A_\varepsilon\right)^2 dv.
\end{equation}
By combining equations \eqref{Lemma16-Eq4} and \eqref{Lemma16-Eq5} we get
\begin{equation}\label{Lemma16-Eq6}
\dfrac{1}{2}\left( \partial_t \int_V f^2_\varepsilon dv + \partial_x \int_V v f^2_\varepsilon dv \right) \leq - \mu_2 \gamma^2 \left( \dfrac{J_\varepsilon n_\varepsilon}{|V|}-\int_V v f^2_\varepsilon dv \right)\alpha (S_\varepsilon).
\end{equation}
Moreover, we have
\begin{eqnarray*}
- \mu_2 \gamma^2 \left( \dfrac{J_\varepsilon n_\varepsilon}{|V|}-\int_V v f^2_\varepsilon dv \right)\alpha (S_\varepsilon) & =& \mu_2\gamma^2 \alpha (S_\varepsilon) \left( \int_V v f^2_\varepsilon dv - \dfrac{J_\varepsilon n_\varepsilon}{|V|} \right)\\
& \leq & \mu_2 \gamma^2 \nu \alpha_\infty \left( \int_V  f^2_\varepsilon dv + \dfrac{ n^2_\varepsilon}{|V|} \right),
\end{eqnarray*}
and using \eqref{Lemma16-Eq5}, we obtain
\begin{eqnarray}\label{Lemma17-Eq7}
- \mu_2 \gamma^2 \left( \dfrac{J_\varepsilon n_\varepsilon}{|V|}-\int_V v f^2_\varepsilon dv \right)\alpha (S_\varepsilon)& \leq & 2 \mu_2 \gamma^2 \nu  \alpha_\infty \int_V  f^2_\varepsilon dv.
\end{eqnarray}
Hence, from \eqref{Lemma16-Eq6} and \eqref{Lemma17-Eq7} we get
\begin{equation*}
\partial_t \int_V  f^2_\varepsilon dv + \partial_x \int_V v f^2_\varepsilon dv \leq C \int_V  f^2_\varepsilon dv,
\end{equation*}
and integration over $x\in \mathbb{R}^d$ yields
\begin{equation}\label{Lemma16-Eq8}
\frac{d}{dt} \|f_\varepsilon(t)\|^2_{L^2_{x,v}}\leq C \|f_\varepsilon(t)\|^2_{L^2_{x,v}}.
\end{equation}
To derive a similar  estimate for $g_\varepsilon$ we multiply the second equation of system \eqref{cks} by $g_\varepsilon$ and we integrate over $V$ to obtain
\begin{equation*}
\dfrac{1}{2}\left(\partial_t \int_V g^2_\varepsilon dv + \partial_x \int_V v g^2_\varepsilon dv \right) =  \sigma\left( \dfrac{S^2_\varepsilon}{|V|}-\int_V g^2_\varepsilon dv \right)+ a n_\varepsilon S_\varepsilon - b S^2_\varepsilon.
\end{equation*}
Using the Cauchy-Schwarz inequality we can write
\begin{equation*}
 \dfrac{1}{2}\left(\partial_t \int_V g^2_\varepsilon dv + \partial_x \int_V v g^2_\varepsilon dv \right) \leq \dfrac{a |V|}{2}\int_V f^2_\varepsilon dv + \left(\dfrac{a}{2}+b \right)|V|\int_V g^2_\varepsilon dv,
\end{equation*}
and integration over the space variable $x \in \mathbb{R}$ gives
\begin{equation}\label{Lemma16-Eq9}
\frac{d}{dt} \|g_\varepsilon(t)\|^2_{L^2_{x,v}} \leq a|V| \|f_\varepsilon(t)\|^2_{L^2_{x,v}} + (a+2b)|V|\|g_\varepsilon(t)\|^2_{L^2_{x,v}}.
\end{equation}
Let us now combine equations \eqref{Lemma16-Eq8} and \eqref{Lemma16-Eq9} to get
$$
\frac{d}{dt} \left[ \|f_\varepsilon (t)\|^2_{L^2_{x,v}} + \|g_\varepsilon (t)\|^2_{L^2_{x,v}} \right] \leq C \left[ \|f_\varepsilon (t)\|^2_{L^2_{x,v}} + \|g_\varepsilon (t)\|^2_{L^2_{x,v}} \right].
$$
We conclude the proof thanks to a Gronwall's inequality.
\end{proof}

\subsection{Convergence by compactness}\label{sec:conv}

According to Lemma \ref{L2-a-priori estimate}, the sequences $f_\varepsilon$, $g_\varepsilon$ are bounded in $L^\infty\left(0,T;L^2_{x,v}\right)$, hence  there are bounded in $L^2_{t,x,v}$. Accordingly, it follows that there exist two subsequences, denoted in the same way, and $f$, $g \in L^2_{t,x,v}$ such that
\begin{equation}
f_\varepsilon \rightharpoonup f, \quad g_\varepsilon \rightharpoonup g \quad \text{in} \;\; L^2_{t,x,v}.
\end{equation}
Moreover, we have
\begin{equation}\label{eq.4.22}
\partial_t g_\varepsilon + v \partial_x g_\varepsilon = \sigma \left( \frac{S_\varepsilon}{\vert V \vert} -g_\varepsilon \right) + a n_\varepsilon - b S_\varepsilon \in L^2_{x,v}.
\end{equation}
Hence, according to averaging Lemma, see for instance \cite{[S09]} Proposition 3.3.1, we have
\begin{equation}
\int_V g_\varepsilon(v) \, dv = S_\varepsilon \mbox{ is uniformly bounded in } L^2\left(0,T; H^{\frac{1}{2}}(\mathbb{R})\right).
\end{equation}

Integrating equation \eqref{eq.4.22} with respect to $v$, we deduce clearly that $\partial_t S_\varepsilon \in L^2\left(0,T;W^{-1,1}(\mathbb{R})\right)$. Moreover, for each compact $K\subset \mathbb{R}$, we have the embeddings (see e.g. \cite{[BCD11]})
\begin{equation}
H^{\frac{1}{2}}(K) \xhookrightarrow[compact]{} L^2(K) \xhookrightarrow[{\color{white}{---}}]{} W^{-1,1}(K).
\end{equation}
From Aubin-Lions compactness Lemma (see \cite{[S85]}), we deduce that the sequence $(S_\varepsilon)_\varepsilon$ is relatively compact in $L^2\left(0,T;L^2(K)\right)$. Hence we can extract a subsequence, still denoted $(S_\varepsilon)_\varepsilon$, which converges strongly towards $S$ in $L^2 \left((0,T)\times K\right)$. By uniqueness of the weak limit, we have that $S = \int_V g(v) dv$. 

However the convergence is global:
\begin{equation}
S_\ve \to S \quad \text{in} \;\; L^2_{t,x}.
\end{equation}
 Indeed, for any compact $[-R,R] \subset\R$ we may extract a subsequence $(S_\ve)_\ve$ such that
$S_\ve \to S$ strongly in $L^2([0,T]\times [-R,R])$, and we know that $S_\ve = \int_V g_\ve(v)\,dv$ where
$$
\partial_t \int_V (f_\ve^2 + g_\ve^2) dv + \partial_x \int_V v (f_\ve^2+g_\ve^2) dv \leq C \int_V (f_\ve^2+g_\ve^2) dv.
$$
Multiplying by a function $x\mapsto\phi(x)\in C^1(\R)$ with bounded derivative and integrating, we deduce
\begin{equation}\label{estimphi1}
\frac{d}{dt} \int_{\R}\int_V (f_\ve^2+g_\ve^2) \phi \,dxdv \leq C \int_{\R}\int_V (f_\ve^2+g_\ve^2) \phi \,dxdv + \int_{\R} \int_V v(f_\ve^2+g_\ve^2)  \phi' \,dxdv.
\end{equation}

In order to pass from local to global convergence, we need to prove that we have a bound on the tail at infinity.
Let us show that $(S_\ve)_\ve$ is a Cauchy sequence in $L^2_{t,x}$.
We compute
$$
\int_0^T \int_{\R} |S_\ve - S_{\ve'}|^2\,dxdt = 
\int_0^T \int_{-R}^R |S_\ve - S_{\ve'}|^2\,dxdt +
\int_0^T \int_{\R\setminus [-R,R]} |S_\ve - S_{\ve'}|^2\,dxdt.
$$
From the above result, we know that the first term of the right hand side 
goes to $0$ as $\ve, \ve' \to 0$. 
For the second term, let us consider $\phi\in C^\infty(\R^d)$ such that $0\leq \phi\leq 1$, $\phi(x)=0$ for $|x|\leq 1/2$ and $\phi(x)=1$ for $|x|\geq 1$.
We define $\phi_R(x) = \phi(x/R)$.
Then, we have
\begin{align*}
\int_0^T \int_{\R\setminus [-R,R]} |S_\ve - S_{\ve'}|^2\,dxdt & \leq
\int_0^T \int_{\R\setminus [-R,R]} |S_\ve - S_{\ve'}|^2 \phi_R\,dxdt  \\
& \leq 2 \int_0^T \int_{\R} (|S_\ve|^2 + |S_{\ve'}|^2) \phi_R\,dxdt.
\end{align*}
Let us now use estimate \eqref{estimphi1} with $\phi_R$, since $\phi'_R(x) = \frac 1R \phi'(x/R)$, we have
\begin{equation}
\begin{split}
\frac{d}{dt} \int_{\R}\int_V (f_\ve^2+g_\ve^2) \phi_R \,dxdv &\leq C \int_{\R}\int_V (f_\ve^2+g_\ve^2) \phi_R \,dxdv\\
& + \frac{1}{R}\int_{\R} \int_V v(f_\ve^2+g_\ve^2)  \phi'(x/R) \,dxdv.
\end{split}
\end{equation}
Applying a Gronwall Lemma, we deduce that 
$$
\int_{\R}\int_V (f_\ve^2+g_\ve^2) \phi_R \,dxdv \leq e^{CT} 
\left(\int_{\R}\int_V ((f^0)^2+(g^0)^2) \phi_R \,dxdv+ \frac{C\|\phi'\|}{R}\right).
$$
Since the initial data $f^0$ and $g^0$ are given in $L^2_{x,v}$ 
and $\phi_R(x)=0$ on $B_{R/2}$, we deduce that 
the left hand side goes to $0$ as $R\to +\infty$, uniformly with respect to $\ve$.
Thus,
$$
\int_0^T \int_{\R} (|S_\ve|^2 + |S_{\ve'}|^2) \phi_R\,dxdt
\leq |V| \int_0^T\int_{\R}\int_V (f_\ve^2+f_{\ve'}^2) \phi_R \,dxdvdt
$$
goes uniformly to $0$ as $R\to +\infty$.
We conclude that the sequence $(S_\ve)_\ve$ is a Cauchy sequence in $L^2([0,T]\times\R)$.
\qed

\subsection{Proof of Theorem \ref{HLRC}} 

Multiply the first and second equations of system \eqref{cks} by 1 and $v$ respectively, and integrate over $V$ to obtain the following system
\begin{equation}\label{theorem15-Eq1}
\begin{cases}
\displaystyle \partial_t n_\varepsilon + \partial_x J_\varepsilon =0 \vspace{0.25cm}\\
\displaystyle \partial_t J_\varepsilon + \partial_x \int_V v^2 f_\varepsilon \, dv = -\mu_1 J_\varepsilon + \mu_2 \gamma^2 \int_V v^2 f_\varepsilon \,dv\, \alpha(S_\varepsilon)\vspace{0.25cm}\\
\partial_t g_\varepsilon + v\cdot \partial_x g_\varepsilon = \sigma\left(\frac{S_\varepsilon}{|V|}-g_\varepsilon\right)+an_\varepsilon-b S_\varepsilon.
\end{cases}
\end{equation}
We have
\begin{equation}\label{limfg}
f_\varepsilon(t,x,v) \rightharpoonup f(t,x,v) \quad \mbox{ and } \quad
g_\varepsilon(t,x,v) \rightharpoonup g(t,x,v) \quad  \text{in} \;\; L^2_{t,x,v}.
\end{equation}
Therefore, since the set of velocities $V$ is bounded, we deduce
\begin{equation}\label{lim1}
n_\varepsilon(t,x) \rightharpoonup n(t,x), \quad
J_\varepsilon(t,x) \rightharpoonup J(t,x) \quad  \text{in} \;\; L^2_{t,x},
\end{equation}
\begin{equation}\label{lim2}
S_\varepsilon(t,x) \rightarrow S(t,x), \quad
q_\varepsilon(t,x) \rightharpoonup q(t,x) \quad  \text{in} \;\; L^2_{t,x},
\end{equation}
\begin{equation}\label{lim3}
\int_V v^2f_\varepsilon(t,x,v) dv \rightharpoonup \int_V v^2 f(t,x,v) dv \quad  \text{in} \;\; L^2_{t,x},
\end{equation}
\begin{equation}\label{lim4}
\int_V v^2 g_\varepsilon(t,x,v) dv \rightharpoonup \int_V v^2 g(t,x,v) dv \quad  \text{in} \;\; L^2_{t,x},
\end{equation}
when $\varepsilon$ tends to zero. However, according to Section \ref{sec:conv} we have 
\begin{equation}\label{weak-lim5}
 \alpha(S_\varepsilon(t,x)) \, \int_V v^2 f_\varepsilon(t,x,v) \,dv  \rightharpoonup \alpha(S(t,x)) \, \int_V v^2 f(t,x,v) \,dv \quad  \text{in} \;\; L^2_{t,x}.
\end{equation}

Hence, by passing to limit in \eqref{theorem15-Eq1}, in the sense of distributions, and taking into account Eqs. \eqref{lim1}-\eqref{weak-lim5}, it follows that
\begin{equation}\label{macro-2-Q}
\begin{cases}
\displaystyle \partial_t n + \partial_x J =0\\
\displaystyle \partial_t J + \partial_x \int_V v^2 f dv = -\mu_1 J + \mu_2 \gamma^2 \int_V v^2 f dv \, \alpha(S)   \\
\displaystyle \partial_t g + v\cdot \partial_x g = \sigma\left(\frac{S}{|V|}-g\right)+an-b S.
\end{cases}
\end{equation}

To identify the term $\int_V v^2 f(t,x,v) dv$, we multiply the first equation of system \eqref{cks} by $\varepsilon$ to get
\begin{equation}
\begin{split}
\varepsilon \partial_t f_\varepsilon (t,x,v) + \varepsilon v & \cdot \partial_x f_\varepsilon (t,x,v)\\
 = \mu_0 ( F_{n_\varepsilon, J_\varepsilon}(t,x,v) & -f_\varepsilon(t,x,v) )
+ \varepsilon \mu_1 \left( \dfrac{n_\varepsilon (t,x)}{\vert V \vert}- f_\varepsilon (t,x,v) \right)\\
 - \varepsilon & \mu_2 \gamma^2 \left( \dfrac{J_\varepsilon (t,x)}{\vert V \vert} -v f_\varepsilon (t,x,v) \right) \alpha\left(S_\varepsilon (t,x)\right).
\end{split}
\end{equation}
Then, letting $\varepsilon$ go to zero yields
\begin{equation}
f=F_{n,J}(t,x,v)=\frac{1}{|V|} \left(n+ \gamma^2 J v\right)
\end{equation}
and a simple calculations shows that
\begin{equation}
\int_V v^2 f(t,x,v) dv =\int_V v^2 F_{n,J}(t,x,v) dv =\frac{1}{\gamma^2} n(t,x).
\end{equation}
Using this last equation in system \eqref{macro-2-Q} finishes the proof. \qed

\bibliographystyle{plain}
\bibliography{biblio}

\end{document}